\documentclass[10pt]{article}
 \font\smallit=cmti10

\usepackage{amssymb,latexsym,amsmath,epsfig,amsthm} 

\makeatletter

\renewcommand{\@seccntformat}[1]{\csname the#1\endcsname. }

\makeatother

\newtheorem{theorem}{Theorem}[section]
 \newtheorem{lemma}[theorem]{Lemma}
 
 \newtheorem{proposition}[theorem]{Proposition}
 \newtheorem{corollary}[theorem]{Corollary}
 \newtheorem{definition}[theorem]{Definition}
 \newtheorem{remark}[theorem]{Remark}

\begin{document}
\begin{center}
 {\bf Combinatorial Sums $\sum_{k\equiv r(\mbox{mod }
m)}{n\choose k}a^k$ and Lucas Quotients}
 \vskip 20pt
 {\bf Jiangshuai Yang}\\
 {\smallit Key Laboratory of Mathematics Mechanization, NCMIS, Academy of Mathematics and Systems Science, Chinese Academy of Sciences, Beijing 100190, People's Republic of China}\\
 {\tt yangjiangshuai@amss.ac.cn}\\
 \vskip 10pt
 {\bf Yingpu Deng}\\
 {\smallit Key Laboratory of Mathematics Mechanization, NCMIS, Academy of Mathematics and Systems Science, Chinese Academy of Sciences, Beijing 100190, People's Republic of China}\\
 {\tt dengyp@amss.ac.cn}\\
 \end{center}
 \vskip 30pt
\centerline{\bf Abstract}

\noindent In this paper, we study the combinatorial sum
$$\sum_{k\equiv r(\mbox{mod }m)}{n\choose k}a^k.$$
By studying this sum, we obtain new congruences for Lucas quotients of two infinite families of Lucas sequences. Only for three Lucas sequences, there are such known results.

\pagestyle{myheadings}
 \thispagestyle{empty}
 \baselineskip=12.875pt
 \vskip 30pt

\section{Introduction}
Let $\{F_n\}_{n\geq0}$ be the Fibonacci sequence, i.e.,
$$F_0=0,\quad F_1=1,\quad F_{n+1}=F_n+F_{n-1}\mbox{ for }n\geq1.$$
For example, $F_2=1,F_3=2,F_4=3,F_5=5$, etc. It is well-known that
$$p\mid F_{p-\left(\frac{p}{5}\right)},$$
where $p$ is an arbitrary prime, and $\left(\frac{p}{5}\right)$ is the Legendre symbol. We know that
$$\left(\frac{p}{5}\right)=\left\{\begin{array}{rl}
0,&\mbox{ if }p=5,\\
1,&\mbox{ if }p\equiv\pm1\pmod5,\\
-1,&\mbox{ if }p\equiv\pm2\pmod5.
\end{array}
\right.
$$
For example, we have that $2\mid F_3,3\mid F_4$ and $5\mid F_5$. In 1960 Wall \cite{wa} posed the problem of whether there exists a prime $p$ such that
$$p^2\mid F_{p-\left(\frac{p}{5}\right)}.$$
Up to now this is still open.

An idea related to Wall's problem is to consider the Fibonacci quotient
$$\frac{F_{p-\left(\frac{p}{5}\right)}}{p}.$$
In 1982 Williams \cite{wi} obtained this quotient as
$$\frac{F_{p-\left(\frac{p}{5}\right)}}{p}\equiv\frac{2}{5}\sum_{k=1}^{p-1-\left[\frac{p}{5}\right]}\frac{(-1)^k}{k}\pmod p,$$
where $p\neq5$ is an odd prime, and $\left[\frac{p}{5}\right]$ is the integral part of $\frac{p}{5}$, i.e., the largest integer $\leq\frac{p}{5}$.

We know that the Fibonacci sequence is a special Lucas sequence. In general, let $A,B\in\mathbb{Z}$, the Lucas sequence $\{u_n\}_{n\geq0}$ is defined as
$$u_0=0,u_1=1,u_{n+1}=Bu_{n}-Au_{n-1}{\mbox{ for }}n\geq1.$$
Thus, when $A=-1$ and $B=1$, we get the Fibonacci sequence. Let $D=B^2-4A$ and let $p\nmid A$ be an odd prime. It is well-known that
$$p\mid u_{p-\left(\frac{D}{p}\right)}.$$
So, similarly, we can consider the Lucas quotient
$$\frac{u_{p-\left(\frac{D}{p}\right)}}{p}\pmod p,$$
and we hope that we can obtain some expression as Williams' for Fibonacci quotient.

Williams' method is to consider the sum
$$\sum_{k\equiv r(\mbox{mod }
5)}{p\choose k},$$
where $r$ is an integer and ${p\choose k}$ is the binomial coefficient with the convention ${p\choose k}=0$ for $k<0$ or $k>p$. Williams did not give any explicit formula for this sum, but he used the properties of the sum to deduce his congruence. Along this line, Sun \cite{sun1,sun2,sun3}, Sun \cite{s1995,s2002} and Sun and Sun \cite{ss} studied the sum
$$\sum_{k\equiv r(\mbox{mod }m)}{n\choose k},$$
where $n,m$ and $r$ are integers with $n>0$ and $m>0$. They gave the formulae of the value of the sum for small $m$ and obtained congruences for two new Lucas sequences. One is the Pell sequence $\{P_n\}_{n\geq0}$ which is defined as
$$P_0=0,P_1=1,P_{n+1}=2P_n+P_{n-1}\mbox{ for }n\geq1.$$
Pell sequence is the Lucas sequence with $A=-1$ and $B=2$. Sun's congruence is
$$\frac{P_{p-\left(\frac{2}{p}\right)}}{p}\equiv(-1)^{\frac{p-1}{2}}\sum_{k=1}^{\left[\frac{p+1}{4}\right]}\frac{(-1)^k}{2k-1}\pmod p,$$
where $p$ is an odd prime, see (\cite{sun2} Theorem 2.5). Sun obtained this congruence by studying the above sum with $m=8$. Sun \cite{s1995} also studied the above sum with $m=8$ to deduce a congruence for primes. Sun and Sun \cite{ss} obtained a new congruence for the Fibonacci quotient by studying the above sum with $m=10$.

The second new Lucas sequence is the sequence $\{S_n\}_{n\geq0}$ which is defined as
$$ S_0=0,S_1=1,S_{n+1}=4S_n-S_{n-1}\mbox{ for }n\geq1.$$
This sequence is the Lucas sequence with $A=1$ and $B=4$. Sun \cite{s2002} obtained
$$\sum_{k=1}^{\frac{p-1}{2}}\frac{3^k}{k}\equiv\sum_{k=1}^{\left[\frac{p}{6}\right]}\frac{(-1)^k}{k}\equiv
-6\left(\frac{2}{p}\right)\frac{S_{\overline{p}}}{p}-q_p(2)\pmod p,
$$
where $p>3$ is a prime, $\overline{p}=\frac{p-\left(\frac{3}{p}\right)}{2}$ and $q_p(2)=\frac{2^{p-1}-1}{p}$ is the Fermat quotient of 2 with respect to $p$. See (\cite{s2002} Theorem 3).
Sun obtained this congruence by studying the above sum with $m=12$.

So far, except the above mentioned three Lucas sequences, there is no any known congruence for new Lucas quotients. Noticed that, we do not consider the case where
$D=B^2-4A$ is a perfect square. If $D=B^2-4A$ is a perfect square, then Lucas quotients degenerate to Fermat quotients. In this paper, we study the more general sum
\begin{equation}\label{generalsum}
\sum_{k\equiv r(\mbox{mod }m)}{n\choose k}a^k.
\end{equation}
When $a=1$, this sum is that considered by Williams, Sun and Sun. By studying this sum, we obtain new congruences for Lucas quotients of two infinite families of Lucas sequences, see Theorems \ref{3lucas} and \ref{4lucas}.

Another motivation of this paper is the result in \cite{dp}. Deng and Pan \cite{dp} connected this combinatorial sum with integer factorization for the first time and they proved that, when $n$ is a composite number, for every integer $a$ with gcd$(n,a)=1$, there exists a pair $(m,r)$ of integers such that the sum (\ref{generalsum}) has a nontrivial greatest common divisor with $n$.

Integer factorization is a famous and very important computational problem, and it is the security foundation of the famous public-key cryptosystem RSA \cite{rsa}. So it is worthwhile to make a systematic research of the combinatorial sum for a general $a$.

The paper is organized as follows. We give some necessary preliminaries
in Section 2. We deduce the recurrence relation for the combinatorial sum in Section 3. We consider the calculation of the combinatorial sum and give some applications when $m=3,4,6$ in Sections 4,5,6, respectively.

\section{Preliminaries}
Notational conventions: We denote by $\mathbb{C},\mathbb{R},\mathbb{Z}$ respectively the complex numbers, the reals and the integers. For $x\in\mathbb{R}$, we denote by $[x]$ the integral part of $x$, i.e., the largest integer $\leq x$.

First we recall some well-known facts about Lucas sequences. Let $A,B\in\mathbb{Z}$. We define Lucas sequences $\{u_n\}_{n\geq0}$ and $\{v_n\}_{n\geq0}$ as
$$u_0=0,u_1=1,u_{n+1}=Bu_{n}-Au_{n-1}{\mbox{ for }}n\geq1;$$
$$v_0=2,v_1=B,v_{n+1}=Bv_{n}-Av_{n-1}{\mbox{ for }}n\geq1.$$

The proof of the following two lemmas can be found in \cite{cp}.
\begin{lemma}\label{lucasbasic}
Let $D=B^2-4A$ and $\alpha,\beta$ be the two complex roots of $x^2-Bx+A=0$. Then we have
\begin{align*}
 u_n&=\frac{\alpha^n-\beta^n}{\alpha-\beta},\;\;\;\;v_n=\alpha^n+\beta^n;\\
  v_n&=u_{n+1}-Au_{n-1}=Bu_n-2Au_{n-1}=2u_{n+1}-Bu_n;\\
 Du_n&=v_{n+1}-Av_{n-1}=Bv_n-2Av_{n-1}=2v_{n+1}-Bv_n;\\
 u_{2n}&=u_nv_n,\;\;\;\;u_{2n+1}=u_{n+1}^2-Au_{n}^2;\\
 v_{2n}&=v_n^2-2A^n,\;\;\;\;v_{n}^2-Du_n^2=4A^n.
\end{align*}
\end{lemma}

\begin{lemma}\label{lucascongru}
  Let $\varepsilon=\left(\frac{D}{p}\right)$ and $p\nmid A$ be an odd prime. Then we have
  $$u_{p-\varepsilon}\equiv0\pmod p,\quad u_p\equiv\varepsilon\pmod p.$$
\end{lemma}

\begin{definition}\label{defsum}
{\rm Let $n,m,r$ and $a$ be integers with $n>0$ and $m>0$. We define
  $$\left[\begin{array}{c}n \\ r\\\end{array}\right] _{m}(a):=\sum_{\substack{k=0\\k\equiv r({\mbox{mod }}m)}}^n\binom nk a^k,$$
  where  ${n\choose k}$ is the binomial coefficient with the convention ${n\choose k}=0$ for $k<0$ or $k>n$.
  Let $p$ be an odd prime and $m,r$ and $a$ be integers with $m>0$. We define $$K_{p,m,r}(a):=\sum_{\substack{k=1\\k\equiv r({\mbox{mod }}m)}}^{p-1}\frac{(-a)^k}{k}.$$}
\end{definition}

\begin{lemma}\label{K}
   With notation as above, we have:\\
   (1) $$\sum_{r=0}^{m-1}\left[\begin{array}{c}n \\ r\\\end{array}\right] _{m}(a)=(1+a)^n;$$
  (2) $$\binom {p-1}k\equiv (-1)^k\pmod p \mbox{ for }0\leq k\leq p-1;   $$
  \noindent(3)
   $$ \left[\begin{array}{c}p \\ r\\\end{array}\right] _{m}(a)\equiv\delta_{0\equiv r(m)}+\delta_{p\equiv r(m)}a^p-pK_{p,m,r}(a) \pmod{p^2},$$
   where
   $$\delta_{0\equiv r(m)}=\left\{\begin{array}{ll}
   1,&\mbox{ if  }0\equiv r\pmod m  \mbox{ holds},\\
   0,&\mbox{ otherwise,}
   \end{array}
   \right.
   $$
   and $\delta_{p\equiv r(m)}$ has the similar meaning.
\end{lemma}

\begin{proof}
(1) is obvious.

  (2) If $k=0,\;\binom {p-1}0\equiv (-1)^0\pmod p$.

     If $1\leq k\leq p-1$, $\binom {p-1}k=\frac{(p-1)(p-2)\cdots (p-k)}{k!}\equiv(-1)^k\pmod p.$

     \noindent(3) By (2), we have
     \begin{align*}
       \left[\begin{array}{c}p \\ r\\\end{array}\right] _{m}(a)
       &=\delta_{0\equiv r(m)}+\delta_{p\equiv r(m)}a^p+\sum_{\substack{k=1\\k\equiv r(\mbox{mod }m)}}^{p-1}\binom{p}{k}a^k\\
       &=\delta_{0\equiv r(m)}+\delta_{p\equiv r(m)}a^p+\sum_{\substack{k=1\\k\equiv r(\mbox{mod }m)}}^{p-1}\frac{p}{k}\binom{p-1}{k-1}a^k\\
       &\equiv\delta_{0\equiv r(m)}+\delta_{p\equiv r(m)}a^p-p\sum_{\substack{k=1\\k\equiv r(\mbox{mod }m)}}^{p-1}\frac{(-a)^k}{k}\pmod {p^2}\\
       &=\delta_{0\equiv r(m)}+\delta_{p\equiv r(m)}a^p-pK_{p,m,r}(a)\pmod {p^2}.\\
     \end{align*}
     \end{proof}
     Note that, in the above lemma, (2) is well-known and (3) is a generalization of Lemma 1.1 in \cite{sun1}.
\begin{definition}
{\rm Let $p$ be an odd prime and $x$ an integer with $p\nmid x$.  Fermat quotient is defined as  $q_p(x):=\frac{x^{p-1}-1}{p}$. In the sequent, when we meet $q_p(x)$, we always suppose that $p$ is an odd prime and $p\nmid x$.}
\end{definition}

\begin{lemma}\label{fermatquo} We have:\\
  (1)$$q_p(x)\equiv 2\left(\frac{x}{p}\right)\frac{x^{\frac{p-1}2}-\left(\frac{x}{p}\right)}{p}\pmod p;$$
  (2) $$q_p(xy)\equiv q_p(x)+q_p(y)\pmod p;$$
  (3) $$\sum_{r=0}^{m-1}K_{p,m,r}(a)\equiv aq_p(a)-(a+1)q_p(a+1)\pmod p.$$
\end{lemma}

\begin{proof}
 (1) From $$x^{p-1}-1=\left(x^{\frac{p-1}2}+\left(\frac{x}{p}\right)\right)\left(x^{\frac{p-1}2}-\left(\frac{x}{p}\right)\right)\equiv 2\left(\frac{x}{p} \right)\left(x^{\frac{p-1}2}-\left(\frac{x}{p} \right)\right)\pmod{p^2},$$
 we have
 $$q_p(x)\equiv 2\left(\frac{x}{p} \right)\frac{x^{\frac{p-1}2}-\left(\frac{x}{p} \right)}{p}\pmod p.$$
 (2) From $$(xy)^{p-1}-1=x^{p-1}(y^{p-1}-1)+(x^{p-1}-1),$$ we have$$q_p(xy)\equiv q_p(x)+q_p(y)\pmod p.$$
 (3) From Lemma \ref{K}, we have
  \begin{align*}
 (1+a)^p&=\sum_{r=0}^{m-1}\left[\begin{array}{c}p \\ r\\\end{array}\right] _{m}(a)\\
 &\equiv\sum_{r=0}^{m-1}\left[\delta_{0\equiv r(m)}+\delta_{p\equiv r(m)}a^p-pK_{p,m,r}(a)\right] \pmod{p^2}\\
 &\equiv1+a^p-p\sum_{r=0}^{m-1}K_{p,m,r}(a)\pmod{p^2}.
  \end{align*}
 Thus we obtain
 $$\sum_{r=0}^{m-1}K_{p,m,r}(a)\equiv\frac{1+a^p-(a+1)^p}{p}=aq_p(a)-(a+1)q_p(a+1)\pmod p.$$
\end{proof}
Note that, in the above lemma, (1) is the Lemma 1.2 in \cite{sun1}.

\begin{lemma}\label{quotient-v}
With notation as in Lemmas \ref{lucasbasic} and \ref{lucascongru}. Let $p$ be an odd prime with $p\nmid DA$. We have:
\begin{description}
  \item[(1)] if $\varepsilon=1$, then $\frac{v_{p-1}-2}{p}\equiv q_p(A)\pmod p$;
 \item[(2)] if $\varepsilon=-1$, then $\frac{v_{p+1}-2A}{p}\equiv Aq_p(A)\pmod p$.
\end{description}
 \end{lemma}
\begin{proof}
  By Lemma \ref{lucasbasic}, we have $v_n^2-Du_n^2=4A^n$ for $n\geq0$. Combining this with Lemma \ref{lucascongru}, we have $v_{p-\varepsilon}^2\equiv4A^{p-\varepsilon}\pmod{p^2}$. Since it is impossible that $p\mid v_{p-\varepsilon}+2A^{\frac{p-\varepsilon}{2}}$
  and $p\mid v_{p-\varepsilon}-2A^{\frac{p-\varepsilon}{2}}$, we have
  \[v_{p-\varepsilon}\equiv\pm2\left(\frac{A}{p}\right)A^{\frac{p-\varepsilon}{2}}\pmod{p^2}.\]
  If $\varepsilon=1$, by Lemmas \ref{lucasbasic} and \ref{lucascongru} we have $v_{p-1}=2u_p-Bu_{p-1}\equiv2\pmod p$. Thus, by Lemma \ref{fermatquo}(1), we have
  $$v_{p-1}\equiv2\left(\frac{A}{p}\right)A^{\frac{p-1}{2}}\equiv2+pq_p(A)\pmod{p^2}.$$
  Hence
   $$\frac{v_{p-1}-2}{p}\equiv q_p(A)\pmod p,$$
   we obtain (1).

  If $\varepsilon=-1$, by Lemmas \ref{lucasbasic} and \ref{lucascongru} we have $v_{p+1}=Bu_{p+1}-2Au_{p}\equiv2A\pmod p$. Thus, by Lemma \ref{fermatquo}(1), we have
  $$v_{p+1}\equiv2\left(\frac{A}{p}\right)A^{\frac{p+1}{2}}\equiv2A+Apq_p(A)\pmod{p^2}.$$
  Hence
   $$\frac{v_{p+1}-2A}{p}\equiv Aq_p(A)\pmod p,$$
   we obtain (2).
\end{proof}
\section{The Recurrence Relation for $\Delta_m(r,n)$}
In this section, we consider the calculation of the sum $\left[\begin{array}{c}n \\ r\\\end{array}\right] _{m}(a)$.
It is easy to see that
  $$\left[\begin{array}{c}n \\ 0\\\end{array}\right] _{1}(a)=(1+a)^n,\quad\left[\begin{array}{c}n \\ r\\\end{array}\right] _{2}(a)=\frac{(1+a)^n+(-1)^r(1-a)^n}{2}.$$
  However, for $m\geq3$, the calculation is much more difficult.

Throughout the rest of this paper, we fix $a\neq0,\pm1$. For any positive integer $m$, let $\zeta_m=e^{2\pi i/m}\in\mathbb{C}$ be the primitive
$m$-th root of unity.

The following lemma is useful to compute $\left[\begin{array}{c}n\\r\end{array}\right]_{m}(a)$ when $m$ is small.

 \begin{lemma}\label{basicsum}
 We have $$\left[\begin{array}{c}n \\ r\\\end{array}\right] _{m}(a)=\frac 1m\sum\limits_{l=0}^{m-1}\zeta_m^{-rl}(1+a\zeta_m^l)^n.$$
 \end{lemma}

   \begin{proof} Since
     \begin{align*}
      \left[\begin{array}{c}n \\ r\\\end{array}\right] _{m}(a)
      &=\sum\limits_{k=0}^n\binom nk a^k\cdot\frac 1m\sum\limits_{l=0}^{m-1}\zeta_m^{(k-r)l}\\
      &=\frac 1m\sum\limits_{l=0}^{m-1}\zeta_m^{-rl}\sum\limits_{k=0}^n\binom nk(a\zeta_m^l)^k\\
      &=\frac 1m\sum\limits_{l=0}^{m-1}\zeta_m^{-rl}(1+a\zeta_m^l)^n,
     \end{align*}
     the lemma follows.
   \end{proof}

   \begin{proposition}\label{polyG}
Let $G_n(x)=\prod\limits_{l=1}^{n}(x-1-a\zeta_{2n+1}^l)(x-1-a\zeta_{2n+1}^{-l}):=\sum\limits_{s=0}^{2n}b_sx^s$ for $n\geq0$.
Then $G_n(x)\in\mathbb{Z}[x]$ and we have:
\begin{description}
  \item[(1)] $b_0=\frac{a^{2n+1}+1}{a+1},b_{2n}=1$ and $b_{s-1}-(a+1)b_s=\binom{2n+1}{s}(-1)^{s+1}$ for $1\leq s\leq 2n-1;$
  \item[(2)] $G_0(x)=1$ and $G_{n+1}(x)=(x-1)^2G_{n}(x)+a^{2n+1}(x+a-1)$ for $n\geq0.$
\end{description}
\end{proposition}

\begin{proof}
Since $(x-1-a)G_n(x)=(x-1)^{2n+1}-a^{2n+1}$, we have
$$\sum_{s=1}^{2n+1}b_{s-1}x^s-\sum_{s=0}^{2n}(1+a)b_sx^s=x^{2n+1}+\sum_{s=1}^{2n}{2n+1\choose s}(-1)^{s+1}x^s-1-a^{2n+1}.$$
By comparing coefficients of both sides of the above expression, we obtain (1).

   Since $(x-1-a)G_n(x)=(x-1)^{2n+1}-a^{2n+1},$ then
   $(x-1-a)G_{n+1}(x)+a^{2n+3}=(x-1)^{2n+3}=(x-1)^2\left[(x-1-a)G_n(x)+a^{2n+1}\right].$
   Hence $G_{n+1}(x)=(x-1)^2G_n(x)+a^{2n+1}(x+a-1)$, thus we obtain (2). The assertion $G_n(x)\in\mathbb{Z}[x]$ follows from (2).
\end{proof}
The polynomial $G_n(x)$ depends on $a$ and should be denoted as $G_{n,a}(x)$, for the notational simplification, we omit $a$.
The first few values of $G_n(x)$ are: $G_0(x)=1,G_1(x)=x^2+(a-2)x+a^2-a+1,G_2(x)=x^4+(a-4)x^3+(a^2-3a+6)x^2+(a^3-2a^2+3a-4)x+a^4-a^3+a^2-a+1$.

\begin{proposition}\label{polyQ}
Let $Q_n(x)=\prod\limits_{l=1}^{n}(x-1-a\zeta_{2n+2}^l)(x-1-a\zeta_{2n+2}^{-l}):=\sum\limits_{s=0}^{2n}c_sx^s$ for $n\geq0$.
Then $Q_n(x)\in\mathbb{Z}[x]$ and we have:
\begin{description}
  \item[(1)] $c_0=\frac{a^{2n+2}-1}{a^2-1},2c_0+(a^2-1)c_1=2n+2,c_{2n-1}=-2n,c_{2n}=1$ and $c_{s-2}-2c_{s-1}+(1-a^2)c_s=\binom{2n+2}{s}(-1)^s$ for $2\leq s\leq 2n-2$;
  \item[(2)] $Q_0(x)=1$ and $Q_{n+1}(x)=(x-1)^2Q_n(x)+a^{2n+2}$ for $n\geq0$.
\end{description}
\end{proposition}

\begin{proof}
Since $(x-1-a)(x-1+a)Q_n(x)=(x-1)^{2n+2}-a^{2n+2}$, we have
$$\sum_{s=2}^{2n+2}c_{s-2}x^s-\sum_{s=1}^{2n+1}2c_{s-1}x^s+\sum_{s=0}^{2n}(1-a^2)c_sx^s$$
$$=x^{2n+2}+\sum_{s=1}^{2n+1}{2n+2\choose s}(-1)^sx^s+1-a^{2n+2}.$$
By comparing coefficients of both sides of the above expression, we obtain (1).

Since $(x-1-a)(x-1+a)Q_n(x)=(x-1)^{2n+2}-a^{2n+2}$, then
   $(x-1-a)(x-1+a)Q_{n+1}(x)+a^{2n+4}=(x-1)^{2n+4}
   =(x-1)^2\left((x-1-a)(x-1+a)Q_{n}(x)+a^{2n+2}\right).$
   Hence $Q_{n+1}(x)=(x-1)^2Q_{n}(x)+a^{2n+2}$, thus we obtain (2). The assertion $Q_n(x)\in\mathbb{Z}[x]$ follows from (2).
\end{proof}

The first few values of $Q_n(x)$ are: $Q_0(x)=1,Q_1(x)=x^2-2x+1+a^2,Q_2(x)=x^4-4x^3+(a^2+6)x^2-(2a^2+4)x+a^4+a^2+1=(x^2+(a-2)x+a^2-a+1)(x^2-(a+2)x+a^2+a+1)$.

\begin{definition}
 {\rm We define
   $$\Delta_m(r,n):=\begin{cases}m\left[\begin{array}{c}n \\ r\\\end{array}\right] _{m}(a)-(1+a)^n,&\mbox{if }m\mbox{ is odd,}\\
  \\
  m\left[\begin{array}{c}n \\ r\\\end{array}\right] _{m}(a)-(1+a)^n-(-1)^{r}(1-a)^n,&\mbox{if }m \mbox{ is even.}\end{cases}$$}
  \end{definition}

  \begin{remark}
  It is obvious that
  $$\Delta_1(r,n)=\Delta_2(r,n)=0.$$
  By Lemma \ref{K}(1), it is easy to see that
  $$\sum_{r=0}^{m-1}\Delta_m(r,n)=0.$$
  \end{remark}

\begin{theorem}\label{oddm}
  Let $m$ be an odd positive integer, and let $G_{\frac{m-1}{2}}(x)=\sum\limits_{s=0}^{m-1}b_sx^s$ be the same as in Proposition \ref{polyG}.
   Then we have
  $\sum\limits_{s=0}^{m-1}b_s\Delta_m(r,n+s)=0$.
\end{theorem}

\begin{proof}
  By Lemma \ref{basicsum}, we have
  \begin{align*}
   \Delta_m(r,n)
   &=m\left[\begin{array}{c}n \\ r\\\end{array}\right] _{m}(a)-(1+a)^n\\
   &=\sum\limits_{l=1}^{m-1}\zeta_m^{-rl}(1+a\zeta_m^l)^n\\
   &=\sum\limits_{l=1}^{\frac{m-1}{2}}\left[\zeta_m^{-rl}(1+a\zeta_m^l)^n+\zeta_m^{rl}(1+a\zeta_m^{-l})^n\right].
  \end{align*}
  Thus
  $$\begin{array}{l}
    \sum\limits_{s=0}^{m-1}b_s\Delta_m(r,n+s)\\
    =\sum\limits_{s=0}^{m-1}b_s\sum\limits_{l=1}^{\frac{m-1}{2}}\left[\zeta_m^{-rl}(1+a\zeta_m^l)^{n+s}+\zeta_m^{rl}(1+a\zeta_m^{-l})^{n+s}\right]\\
    =\sum\limits_{l=1}^{\frac{m-1}{2}}\zeta_m^{-rl}(1+a\zeta_m^{l})^n\sum\limits_{s=0}^{m-1}b_s(1+a\zeta_m^l)^s+
    \sum\limits_{l=1}^{\frac{m-1}{2}}\zeta_m^{rl}(1+a\zeta_m^{-l})^n\sum\limits_{s=0}^{m-1}b_s(1+a\zeta_m^{-l})^s\\
    =\sum\limits_{l=1}^{\frac{m-1}{2}}\zeta_m^{-rl}(1+a\zeta_m^{l})^nG_{\frac{m-1}{2}}(1+a\zeta_m^l)
    +\sum\limits_{l=1}^{\frac{m-1}{2}}\zeta_m^{rl}(1+a\zeta_m^{-l})^nG_{\frac{m-1}{2}}(1+a\zeta_m^{-l})\\
    =0+0\\
    =0.
  \end{array}$$
\end{proof}

\begin{theorem}\label{evenm}
  Let $m$ be an even positive integer, and let
  $Q_{\frac{m}{2}-1}(x)=\sum\limits_{s=0}^{m-2}c_sx^s$ be the same as in Proposition \ref{polyQ}. Then we have
  $\sum\limits_{s=0}^{m-2}c_s\Delta_m(r,n+s)=0$.
\end{theorem}

\begin{proof}
  By Lemma \ref{basicsum}, we have
  \begin{align*}
   \Delta_m(r,n)
   &=m\left[\begin{array}{c}n \\ r\\\end{array}\right] _{m}(a)-(1+a)^n-(-1)^{r}(1-a)^n\\
   &=\sum_{\substack{l=1\\l\neq\frac m2}}^{m-1}\zeta_m^{-rl}(1+a\zeta_m^{l})^n\\
   &=\sum\limits_{l=1}^{\frac{m}{2}-1}\left[\zeta_m^{-rl}(1+a\zeta_m^{l})^n+\zeta_m^{rl}(1+a\zeta_m^{-l})^n\right].
  \end{align*}
  Thus
  $$\begin{array}{l}
  \sum\limits_{s=0}^{m-2}c_s\Delta_m(r,n+s)\\
  =\sum\limits_{s=0}^{m-2}c_s\sum\limits_{l=1}^{\frac{m}{2}-1}\left[\zeta_m^{-rl}(1+a\zeta_m^{l})^{n+s}+\zeta_m^{rl}(1+a\zeta_m^{-l})^{n+s}\right]\\
  =\sum\limits_{l=1}^{\frac{m}{2}-1}\zeta_m^{-rl}(1+a\zeta_m^{l})^{n}\sum\limits_{s=0}^{m-2}c_s(1+a\zeta_m^{l})^{s}
    +\sum\limits_{l=1}^{\frac{m}{2}-1}\zeta_m^{rl}(1+a\zeta_m^{-l})^{n}\sum\limits_{s=0}^{m-2}c_s(1+a\zeta_m^{-l})^{s}\\
    =\sum\limits_{l=1}^{\frac{m}{2}-1}\zeta_m^{-rl}(1+a\zeta_m^{l})^{n}Q_{\frac{m}{2}-1}(1+a\zeta_m^{l})+
    \sum\limits_{l=1}^{\frac{m}{2}-1}\zeta_m^{rl}(1+a\zeta_m^{-l})^{n}Q_{\frac{m}{2}-1}(1+a\zeta_m^{-l})\\
    =0+0\\
    =0.
  \end{array}$$

\end{proof}

\section{$\Delta_3(r,n)$ and Related Lucas Quotients}

In this section, we consider the calculation of $\Delta_3(r,n)$.

\subsection{General Properties}

\begin{theorem}\label{m3}
  Let $\{u_n\}_{n\geq0}$ be the Lucas sequence defined as $$u_0=0,\;u_1=1,\;u_{n+1}=(2-a)u_n-(a^2-a+1)u_{n-1}\mbox{ for }n\geq1.$$ Then we have, for $n\geq1$,
  \begin{align*}
  \Delta_3(0,n)&=(2-a)u_n-2(a^2-a+1)u_{n-1}=2u_{n+1}-(2-a)u_n,\\
  \Delta_3(1,n)&=(2a-1)u_n+(a^2-a+1)u_{n-1}=-u_{n+1}+(a+1)u_n,\\
  \Delta_3(2,n)&=(-a-1)u_n+(a^2-a+1)u_{n-1}=-u_{n+1}-(2a-1)u_n.
  \end{align*}
\end{theorem}

\begin{proof}
By Lemma \ref{lucasbasic}, we have, for $n\geq1$,
\begin{align*}
 (2-a)u_n-2(a^2-a+1)u_{n-1}&=2u_{n+1}-(2-a)u_n,\\
  (2a-1)u_n+(a^2-a+1)u_{n-1}&=-u_{n+1}+(a+1)u_n,\\
  (-a-1)u_n+(a^2-a+1)u_{n-1}&=-u_{n+1}-(2a-1)u_n.
  \end{align*}
  Since $u_2=2-a$, one can verify the following simple facts:
  \begin{align*}
   \Delta_3(0,1)&=-a+2=(2-a)u_1-2(a^2-a+1)u_{0},\\
   \Delta_3(0,2)&=-a^2-2a+2=(2-a)u_2-2(a^2-a+1)u_{ 1},\\
   \Delta_3(1,1)&=2a-1=(2a-1)u_1+(a^2-a+1)u_{0},\\
   \Delta_3(1,2)&=-a^2+4a-1=(2a-1)u_2+(a^2-a+1)u_{1},\\
   \Delta_3(2,1)&=-a-1=(-a-1)u_1+(a^2-a+1)u_{0},\\
  \Delta_3(2,2)&=2a^2-2a-1=(-a-1)u_2+(a^2-a+1)u_{1}.
 \end{align*}
 By Theorem \ref{oddm}, we have $$ \Delta_3(r,n+2)=(2-a)\Delta_3(r,n+1)-(a^2-a+1)\Delta_3(r,n)\mbox{ for }n\geq1.$$
 Then we can prove the theorem by induction on $n$.
 \end{proof}

 \begin{remark}\label{m3v}
 Let $\{v_n\}_{n\geq0}$ be the Lucas sequence defined as $$v_0=2,\;v_1=2-a,\;v_{n+1}=(2-a)v_n-(a^2-a+1)v_{n-1}\mbox{ for }n\geq1.$$ Then, by the above theorem and Lemma \ref{lucasbasic}, we have, for $n\geq1$,
  \begin{align*}
  \Delta_3(0,n)&=v_n,\\
  \Delta_3(1,n)&=-\frac{1}{a}v_n+\frac{a^2-a+1}{a}v_{n-1},\\
  \Delta_3(2,n)&=-\frac{a-1}{a}v_n-\frac{a^2-a+1}{a}v_{n-1}.
  \end{align*}
 \end{remark}

\begin{theorem}\label{3lucas}
  Let $p\nmid3a(a^3+1)$ be an odd prime, $\{u_n\}_{n\geq0}$ as in Theorem \ref{m3}, and
   $K_{p,3,r}(a)$ as in Definition \ref{defsum}.  Then we have:

    (1) for $p\equiv1\pmod 3$,
    \begin{align*}
\frac{u_{p-1}}{p}\equiv&\frac{(2a-1)K_{p,3,0}(a)+(a-2)K_{p,3,1}(a)}{a(a^2-a+1)}\\
&-\frac{a-2}{a^2-a+1}q_p(a)+\frac{a^2-1}{a(a^2-a+1)}q_p(a+1)\pmod p;
\end{align*}
(2) for $p\equiv2\pmod 3$,
$$\frac{u_{p+1}}{p}\equiv\frac{(a-2)K_{p,3,1}(a)-(a+1)K_{p,3,0}(a)}{a}-\frac{a+1}aq_p(a+1)\pmod p.$$
\end{theorem}

\begin{proof}
Since $(2-a)^2-4(a^2-a+1)=-3a^2$, by Lemma \ref{lucascongru}, we have $$p\mid u_{p-\left(\frac{-3}{p}\right)}.$$
  By Theorem \ref{m3}, we have
  \begin{equation}\label{3p-1}
  (2a-1)\Delta_3(0,p)+(a-2)\Delta_3(1,p)=-3a(a^2-a+1)u_{p-1},
  \end{equation}
  and
  \begin{equation}\label{3p+1}
  (a+1)\Delta_3(0,p)-(a-2)\Delta_3(1,p)=3au_{p+1}.
  \end{equation}
  If $p\equiv1\pmod 3$, by Lemma \ref{K}(3), we have
  $$\Delta_3(0,p)\equiv3-3pK_{p,3,0}(a)-(1+a)^p\pmod{p^2},$$
  $$\Delta_3(1,p)\equiv3a^p-3pK_{p,3,1}(a)-(1+a)^p\pmod{p^2}.$$
  Then, by Eq.(\ref{3p-1}),
   \begin{align*}
  a(a^2-a+1)u_{p-1}\equiv& p\left[(2a-1)K_{p,3,0}(a)+(a-2)K_{p,3,1}(a)\right]\\
  &-(a^2-2a)(a^{p-1}-1)+(a^2-1)\left[(a+1)^{p-1}-1\right]\pmod{p^2}.
  \end{align*}
  Thus \begin{align*}
\frac{u_{p-1}}{p}\equiv&\frac{(2a-1)K_{p,3,0}(a)+(a-2)K_{p,3,1}(a)}{a(a^2-a+1)}\\
&-\frac{a-2}{a^2-a+1}q_p(a)+\frac{a^2-1}{a(a^2-a+1)}q_p(a+1)\pmod p.
\end{align*}
  \noindent If $p\equiv2\pmod 3$, by Lemma \ref{K}(3), we have
  $$\Delta_3(0,p)\equiv3-3pK_{p,3,0}(a)-(1+a)^p\pmod{p^2},$$
  $$\Delta_3(1,p)\equiv-3pK_{p,3,1}(a)-(1+a)^p\pmod {p^2}.$$
  Then, by Eq.(\ref{3p+1}),
  $$au_{p+1}\equiv p\left[(a-2)K_{p,3,1}(a)-(a+1)K_{p,3,0}(a)\right]-\left[(a+1)^p-(a+1)\right]\pmod{p^2}.$$
  Thus $$\frac{u_{p+1}}p\equiv\frac{1}{a}\left[(a-2)K_{p,3,1}(a)-(a+1)K_{p,3,0}(a)\right]-\frac{a+1}aq_p(a+1)\pmod p.$$
\end{proof}

Given a value of $a$, by Theorem \ref{3lucas}, we can obtain a concrete congruence for a specific Lucas quotient. We provide one such example.

\begin{corollary}
  Let $\{u_n\}_{n\geq0}$ be the Lucas sequence defined by $$u_0=0,\;u_1=1,\;u_{n+1}=4u_n-7u_{n-1}\mbox{ for }n\geq1,$$ and $p\neq3,7$ be an odd prime.
  Then we have
  $$\frac{u_{p-1}}{p}\equiv\frac{5}{42}\sum\limits_{k=1}^{\frac{p-1}{3}}\frac{8^k}{k}+\frac{1}{14}\sum\limits_{k=1}^{\frac{p-1}{3}}\frac{8^k}{3k-2}
  +\frac{4}{7}q_p(2)
  \pmod p,\mbox{ if }p\equiv1\pmod 3;$$
  $$\frac{u_{p+1}}{p}\equiv\frac{1}{2}\sum\limits_{k=1}^{\frac{p+1}{3}}\frac{8^k}{3k-2}
  -\frac{1}{6}\sum\limits_{k=1}^{\frac{p-2}{3}}\frac{8^k}{k}\pmod p,\mbox{ if }p\equiv2\pmod 3.$$
\end{corollary}

\begin{proof}
  Set $a=-2$ in Theorem \ref{3lucas}.
\end{proof}

   \subsection{The Case a=2}

   If $a=2$, by Theorem \ref{oddm}, we have $\Delta_3(r,n+2)=-3\Delta_3(r,n)$ for $n\geq1$. Thus we have a refinement of Theorem \ref{m3}.

\begin{theorem}\label{2m3}
  Set $a=2$. Let $\Delta_3(r,n)=3\left[\begin{array}{c}n \\ r\\\end{array}\right] _{3}(2)-3^n$ for $n\geq1$. Then we have:
  if $n$ is odd, $$\Delta_3(0,n)=0,\;\Delta_3(1,n)=-(-3)^{\frac{n+1}{2}},\;\Delta_3(2,n)=(-3)^{\frac{n+1}{2}};$$
  if $n$ is even, $$\Delta_3(0,n)=2\cdot(-3)^{\frac{n}{2}},\;\Delta_3(1,n)=\Delta_3(2,n)=-(-3)^{\frac{n}{2}}.$$
\end{theorem}

\begin{proof}
  Since $a=2$, by Theorem \ref{oddm}, we have $\Delta_3(r,n+2)=-3\Delta_3(r,n)$ for $n\geq1$.
  One can verify that
   $$\Delta_3(0,1)=0,\;\Delta_3(1,1)=3,\;\Delta_3(2,1)=-3;$$
  $$\Delta_3(0,2)=-6,\;\Delta_3(1,2)=3,\;\Delta_3(2,2)=3.$$
  Then we can prove the theorem by induction on $n$.
\end{proof}

\begin{remark}
  Using the above theorem and without the use of the Quadratic Reciprocity Law, for an odd prime $p>3$, we can get the Legendre symbol
  $$ \left(\frac{-3}{p}\right)=\left\{
  \begin{array}{rl}
  1,&\mbox{if }p\equiv1\pmod 3,\\
  -1,&\mbox{if }p\equiv2\pmod 3.\end{array}\right.$$
\end{remark}

\begin{proof}
If $p\equiv1\pmod3$, by Theorem \ref{2m3}, we have $\Delta_3(1,p)=-(-3)^{\frac{p+1}{2}}$. Since
\begin{align*}
\Delta_3(1,p)=3\left[\begin{array}{c}p \\ 1\\\end{array}\right] _{3}(2)-3^p=3\sum_{\substack{k=0\\k\equiv 1({\mbox{mod }}3)}}^p{p\choose k}2^k-3^p\\
\equiv3\cdot2^p-3^p\equiv3\cdot2-3=3\pmod p,
\end{align*}
we have $(-3)^{\frac{p-1}{2}}\equiv1\pmod p$. Hence $\left(\frac{-3}{p}\right)=1$.

If $p\equiv2\pmod3$, by Theorem \ref{2m3}, we have $\Delta_3(2,p)=(-3)^{\frac{p+1}{2}}$. Similarly, we can obtain $\left(\frac{-3}{p}\right)=-1$.

\end{proof}

\begin{corollary}
  Let $p>3$ be an odd prime. Then we have $$\sum\limits_{k=1}^{[\frac{p}{3}]}\frac{(-8)^k}{k}\equiv-3q_p(3)\pmod p.$$
\end{corollary}

\begin{proof}
  By Theorem \ref{2m3}, we have $$\left[\begin{array}{c}p\\ 0\\\end{array}\right] _{3}(2)=3^{p-1}.$$
  By Lemma \ref{K}(3), we have
  \begin{align*}
   \sum\limits_{k=1}^{[\frac{p}{3}]}\frac{(-8)^k}{k}=3K_{p,3,0}(2)\equiv3\frac{1-\left[\begin{array}{c}p\\ 0\\\end{array}\right] _{3}(2)}{p}\\
   =3\frac{1-3^{p-1}}{p}=-3q_p(3)\pmod p.
   \end{align*}
 \end{proof}

 \begin{corollary}
   Let $p>3$ be an odd prime. Then we have
    $$\sum\limits_{k=1}^{[\frac{p+1}{3}]}\frac{(-8)^k}{12k-8}+\sum\limits_{k=1}^{[\frac{p}{3}]}\frac{(-8)^k}{6k-2}\equiv
    \left(\frac{-3}{p}\right)\left(2q_p(2)-q_p(3)\right)\pmod p.$$
 \end{corollary}

 \begin{proof}
   By Theorem \ref{2m3},
    $$\left[\begin{array}{c}p\\ 1\\\end{array}\right] _{3}(2)=3^{p-1}+(-3)^{\frac{p-1}{2}},
    \left[\begin{array}{c}p\\ 2\\\end{array}\right] _{3}(2)=3^{p-1}-(-3)^{\frac{p-1}{2}}.  $$
    Then $$\left[\begin{array}{c}p\\ 1\\\end{array}\right] _{3}(2)-\left[\begin{array}{c}p\\ 2\\\end{array}\right] _{3}(2)=2\cdot(-3)^{\frac{p-1}{2}}.$$
     By Lemmas \ref{K} and \ref{fermatquo}, we have
    \begin{align*}
    &\quad  \sum\limits_{k=1}^{[\frac{p+1}{3}]}\frac{(-8)^k}{12k-8}+\sum\limits_{k=1}^{[\frac{p}{3}]}\frac{(-8)^k}{6k-2}\\
    &=K_{p,3,1}(2)-K_{p,3,2}(2)\\
    &\equiv\frac{2^p\cdot\left(\frac{-3}{p}\right)-2\cdot(-3)^{\frac{p-1}{2}}}{p}\\
    &=\left(\frac{-3}{p}\right)\frac{(2^p-2)+\left(2-2(-3)^{\frac{p-1}{2}}\left(\frac{-3}{p}\right)\right)}{p}\\
    &\equiv\left(\frac{-3}{p}\right)(2q_p(2)-q_p(-3))\pmod p.
    \end{align*}

 \end{proof}

 \subsection{Further Results}

 \begin{lemma}\label{vp/p}
     Let $p$ be an odd prime with $p\nmid 3a(2-a)(a^2-a+1)$, and let $\{v_n\}_{n\geq0}$ be the Lucas sequence defined as
     \[v_0=2,v_1=2-a, v_{n+1}=(2-a)v_n-(a^2-a+1)v_{n-1}\;\textup{for}\;n\geq1.\]
     Then we have
     $$\frac{v_p -(2-a)}{p}\equiv-\sum\limits_{k=1}^{[\frac p3]}\frac{(-a)^{3k}}{k}-(a+1)q_p(a+1)\pmod p.$$
   \end{lemma}

   \begin{proof}
   Let $\{u_n\}_{n\geq0}$ be the Lucas sequence defined as
  \[u_0=0,u_1=1,u_{n+1}=(2-a)u_n-(a^2-a+1)u_{n-1}\;\textup{for}\;n\geq1.\]
   By Lemmas \ref{lucasbasic} and \ref{lucascongru}, we have $v_p=(2-a)u_p-2(a^2-a+1)u_{p-1}=2u_{p+1}-(2-a)u_p\equiv2-a\pmod p$.
   By Remark \ref{m3v} and Lemma \ref{K}, we have $v_p=\Delta_3(0,p)=3\left[\begin{array}{c}p \\ 0\\\end{array}\right] _{3}(a)-(1+a)^p\equiv3-3pK_{p,3,0}(a)-(1+a)^p\pmod {p^2}$. Thus
     \[\frac{v_p-(2-a)}{p}\equiv-\sum\limits_{k=1}^{[\frac p3]}\frac{(-a)^{3k}}{k}-(a+1)q_p(a+1)\pmod p.\]
   \end{proof}

   In Theorem \ref{3lucas}, when we express the Lucas quotient, it involves two $K$'s, i.e., two sums. The following theorem can reduce the Lucas quotient to one sum.

\begin{theorem}\label{3lucas+}
  Let $p$ be an odd prime with $p\nmid 3a(2-a)(a^3+1)$, and let $\{u_n\}_{n\geq0}$ be the Lucas sequence defined as
  \[u_0=0,u_1=1,u_{n+1}=(2-a)u_n-(a^2-a+1)u_{n-1}\;\textup{for}\;n\geq1.\]
   Then we have, if $p\equiv1\pmod3$,
 \begin{align*}
  \frac {u_{p-1}}p\equiv&
  \frac{2}{3a^2} \sum\limits_{k=1}^{ \frac {p-1}3 }\frac{(-a)^{3k}}{k}\\
  & +\frac{1}{3a^2}\left((2-a)q_p(a^2-a+1)+2(a+1)q_p(a+1)\right) \pmod p;
\end{align*}
if $p\equiv2\pmod3$,
\begin{align*}
  \frac {u_{p+1}} p\equiv&
 -\frac{2(a^2-a+1)}{3a^2} \sum\limits_{k=1}^{\frac {p-2}3}\frac{(-a)^{3k}}{k}\\
  &-\frac{a^2-a+1}{3a^2}\left((2-a)q_p(a^2-a+1)+2(a+1)q_p(a+1)\right) \pmod p.
\end{align*}
\end{theorem}

 \begin{proof}
 Let $\{v_n\}_{n\geq0}$ be the sequence as in Lemma \ref{vp/p}.
   If $p\equiv1\pmod3$, by Lemma \ref{lucasbasic}, we have $u_{p-1}=\frac{1}{3a^2}((2-a)v_{p-1}-2v_p)$. Thus
   \[
   \frac{u_{p-1}}{p}=\frac{1}{3a^2}\left(\frac{(2-a)(v_{p-1}-2)}{p}-2\frac{v_p-(2-a)}{p}\right).
   \]
   If $p\equiv2\pmod3$, by Lemma \ref{lucasbasic}, we have $u_{p+1}=\frac{1}{3a^2}(2(a^2-a+1)v_{p}-(2-a)v_{p+1})$. Thus
   \[
   \frac{u_{p+1}}{p}=\frac{1}{3a^2}\left(\frac{2(a^2-a+1)(v_{p}-(2-a))}{p}-\frac{(2-a)(v_{p+1}-2(a^2-a+1))}{p}\right).
   \]
   Thus by Lemmas \ref{quotient-v} and \ref{vp/p}, we can prove this theorem.
 \end{proof}

 Given a value of $a$, by Theorem \ref{3lucas+}, we can obtain a concrete congruence for a specific Lucas quotient. We provide one such example.

   \begin{corollary}
   Let $p\neq3,7$ be an odd prime, and $\{u_n\}_{n\geq0}$ be the Lucas sequence defined as
   \[u_0=0,u_1=1,u_{n+1}=4u_n-7u_{n-1}\;\textup{for}\;n\geq1.\]
   Then we have, if $p\equiv1\pmod3$,
   \[\frac{u_{p-1}}{p}\equiv\frac16\sum\limits_{k=1}^{\frac{p-1}{3}}\frac{8^k}k +\frac13q_p(7) \pmod p;\]
   if $p\equiv2\pmod3$,
          \[\frac{u_{p+1}}{p}\equiv-\frac76 \sum\limits_{k=1}^{\frac{p-2}{3}}\frac{8^k}k -\frac73q_p(7) \pmod p. \]
  \end{corollary}

  \begin{proof}
    Set $a=-2$ in Theorem \ref{3lucas+}.
  \end{proof}

\section{$\Delta_4(r,n)$ and Related Lucas Quotients}
In this section, we consider the calculation of $\Delta_4(r,n)$.
\subsection{General Properties}
\begin{theorem}\label{m4}
  Let $\{u_n\}_{n\geq0}$ be the Lucas sequence defined as $$u_0=0,\;u_1=1,\;u_{n+1}=2u_n-(a^2+1)u_{n-1}\mbox{ for }n\geq1.$$ Then we have, for $n\geq1$,
  \begin{align*}
  \Delta_4(0,n)&=2u_n-2(a^2+1)u_{n-1}=2u_{n+1}-2u_n,\\
  \Delta_4(1,n)&=2au_n,\\
  \Delta_4(2,n)&=-2u_n+2(a^2+1)u_{n-1}=-2u_{n+1}+2u_n,\\
  \Delta_4(3,n)&=-2au_n.
  \end{align*}
\end{theorem}

\begin{proof}
Since $u_{n+1}=2u_n-(a^2+1)u_{n-1}$ for $n\geq1$, we have $2u_{n+1}-2u_n=2u_n-2(a^2+1)u_{n-1}$.
It is easy to see that $\Delta_4(r,n)+\Delta_4(r+2,n)=2\Delta_2(r,n)=0$ for $n\geq1$. So we need only consider $\Delta_4(0,n)$ and $\Delta_4(1,n)$.

   Since $u_2=2$, one can verify the following simple facts:
   \begin{align*}
   \Delta_4(0,1)&=2=2u_1-2(a^2+1)u_{0},\\
   \Delta_4(0,2)&=-2a^2+2=2u_2-2(a^2+1)u_{1},\\
   \Delta_4(1,1)&=2a=2au_1,\\
   \Delta_4(1,2)&=4a=2au_2.\\
   \end{align*}
   By Theorem \ref{evenm}, for $n\geq1$, $$\Delta_4(r,n+2)=2\Delta_4(r,n+1)-(a^2+1)\Delta_4(r,n).$$
   Then we can prove the theorem by induction on $n$.
\end{proof}

\begin{theorem}\label{4lucas}
  Let $p\nmid a(a^4-1)$ be an odd prime, $\{u_n\}_{n\geq0}$ as in Theorem \ref{m4}, and
   $K_{p,4,r}(a)$ as in Definition \ref{defsum}. Then we have
    \begin{align*}
    \frac{u_{p-1}}{p}\equiv&\frac{2}{a(a^2+1)}(aK_{p,4,0}(a)-K_{p,4,1}(a))+\frac{2}{a^2+1}q_p(a)\\
    &+\frac{a^2-1}{2a(a^2+1)}(q_p(a+1)-q_p(a-1))\pmod p,\mbox{ if }p\equiv1\pmod4;
 \end{align*}
  \begin{align*}
 \frac{u_{p+1}}{p}\equiv&-\frac{2}{a}(aK_{p,4,0}(a)+K_{p,4,1}(a))-\frac{(a+1)^2}{2a}q_p(a+1)\\
 &+\frac{(a-1)^2}{2a}q_p(a-1)\pmod p,\mbox{ if }p\equiv3\pmod4.
 \end{align*}
\end{theorem}

\begin{proof}
Since $2^2-4(a^2+1)=-4a^2$, by Lemma \ref{lucascongru}, we have
$$p\mid u_{p-\left(\frac{-1}{p}\right)}.$$
By Theorem \ref{m4}, we have
\begin{equation}\label{4p-1}
 \Delta_4(1,p)-a\Delta_4(0,p)=2a(a^2+1)u_{p-1},
 \end{equation}
 and
 \begin{equation}\label{4p+1}
 \Delta_4(1,p)+a\Delta_4(0,p)=2au_{p+1}.
 \end{equation}
Then, by Lemma \ref{K}(3), if $p\equiv1\pmod4)$, we have
 $$\Delta_4(0,p)\equiv4-4pK_{p,4,0}(a)-(1+a)^p-(1-a)^{p}\pmod{p^2},$$
 $$\Delta_4(1,p)\equiv4a^p-4pK_{p,4,1}(a)-(1+a)^p+(1-a)^{p}\pmod{p^2}.$$
 Thus, by Eq.(\ref{4p-1}), we have
 \begin{align*}
 2a(a^2+1)u_{p-1}\equiv&4p(aK_{p,4,0}(a)-K_{p,4,1}(a))+4a(a^{p-1}-1)\\
 &+(a^2-1)\left[(a+1)^{p-1}-(a-1)^{p-1}\right]\pmod{p^2}.
 \end{align*}
 Hence
 \begin{align*}
 \frac{u_{p-1}}{p}\equiv&\frac{2}{a(a^2+1)}(aK_{p,4,0}(a)-K_{p,4,1}(a))+\frac{2}{a^2+1}q_p(a)\\
 &+\frac{a^2-1}{2a(a^2+1)}(q_p(a+1)-q_p(a-1))\pmod p.
  \end{align*}
  If $p\equiv3\pmod4$, by Lemma \ref{K}(3), we have
 $$\Delta_4(0,p)\equiv4-4pK_{p,4,0}(a)-(1+a)^p-(1-a)^{p}\pmod{p^2},$$
 $$\Delta_4(1,p)\equiv-4pK_{p,4,1}(a)-(1+a)^p+(1-a)^{p}\pmod{p^2}.$$
 Thus, by Eq.(\ref{4p+1}), we have
 \begin{align*}
 2au_{p+1}\equiv&-4p(aK_{p,4,0}(a)+K_{p,4,1}(a))-(a+1)^2((a+1)^{p-1}-1)\\
 &+(a-1)^2((a-1)^{p-1}-1)\pmod{p^2}.
 \end{align*}
 Hence
 \begin{align*}
 \frac{u_{p+1}}{p}\equiv&-\frac{2}{a}(aK_{p,4,0}(a)+K_{p,4,1}(a))-\frac{(a+1)^2}{2a}q_p(a+1)\\
 &+\frac{(a-1)^2}{2a}q_p(a-1)\pmod{p}.
 \end{align*}
\end{proof}

Given a value of $a$, by Theorem \ref{4lucas}, we can obtain a concrete congruence for a specific Lucas quotient. We provide one such example.
 \begin{corollary}
   Let $p>5$ be an odd prime, and $\{u_n\}_{n\geq0}$ be the Lucas sequence defined as $$u_0=0,\;u_1=1,\;u_{n+1}=2u_n-5u_{n-1}\mbox{ for }n\geq1.$$
   Then, if $p\equiv1\pmod4$,
   $$\frac{u_{p-1}}p\equiv\frac{1}{10}\sum\limits_{k=1}^{\frac{p-1}{4}}\frac{16^k}{k}+\frac {1}{40}\sum\limits_{k=1}^{\frac{p-1}{4}}\frac{16^k}{4k-3}
   +\frac25q_p(2)+\frac3{20}q_p(3)\pmod p;$$
    if $p\equiv3\pmod4$,
    $$\frac{u_{p+1}}p\equiv\frac{1}{8}\sum\limits_{k=1}^{\frac{p+1}{4}}\frac{16^k}{4k-3}-\frac 12\sum\limits_{k=1}^{\frac{p-3}{4}}\frac{16^k}{k}
   -\frac94q_p(3)\pmod p.$$
 \end{corollary}

 \begin{proof}
   Set $a=-2$ in Theorem \ref{4lucas}.
 \end{proof}

 \subsection{Further Results}

 In Theorem \ref{4lucas}, when we express the Lucas quotient, it involves two $K$'s, i.e., two sums. The following theorem can reduce the Lucas quotient to one sum.

 \begin{theorem}\label{4lucas+}
   With notation as in Theorem \ref{4lucas}. Let $p\nmid a(a^4-1)$ be an odd prime. We have:
   if $p\equiv1\pmod4$, then
   \begin{align*}
  \frac{u_{p-1}}{p}
  \equiv&\frac{1}{2a^2}\left(\sum\limits_{k=1}^{\frac{p-1}{4}}\frac{a^{4k}}{k}+(1+a)q_p(1+a)+(1-a)q_p(1-a)+q_p(a^2+1)\right)\\
  \equiv&\frac{1}{2a}\left(-4\sum\limits_{k=1}^{\frac{p-1}{4}}\frac{a^{4k-1}}{4k-1}+(1+a)q_p(1+a)-(1-a)q_p(1-a)\right)\\
  &-\frac{1}{2}q_p(a^2+1)\pmod p;
 \end{align*}
 if $p\equiv3\pmod 4$, then
  \begin{align*}
  \frac{u_{p+1}}{p}
  \equiv&-\frac{a^2+1}{2a^2}\left(\sum\limits_{k=1}^{\frac{p-3}{4}}\frac{a^{4k}}{k}+(1+a)q_p(1+a)+(1-a)q_p(1-a)+q_p(a^2+1)\right)\\
  \equiv&\frac{a^2+1}{2a}\left(4\sum\limits_{k=1}^{\frac{p+1}{4}}\frac{a^{4k-3}}{4k-3}-(1+a)q_p(1+a)+(1-a)q_p(1-a)\right)\\
  &+\frac{a^2+1}{2}q_p(a^2+1)\pmod p.
 \end{align*}
 \end{theorem}

 \begin{proof}
 Let $\{v_n\}_{n\geq0}$ be the Lucas sequence defined as
 \[v_0=2,v_1=2,v_{n+1}=2v_n-(a^2+1)v_{n-1}\;\textup{for}\;n\geq1.\]
 By Lemmas \ref{lucasbasic} and \ref{lucascongru}, we have $v_p=2u_p-2(a^2+1)u_{p-1}=2u_{p+1}-2u_p\equiv2\pmod p$.
 By Theorem \ref{m4} and Lemma \ref{lucasbasic}, we have $\Delta_4(0,p)=2u_{p+1}-2u_p=v_p$.
 By Lemma \ref{K}, we have
      $\Delta_4(0,p)=4\left[\begin{array}{c}p \\ 0\\\end{array}\right] _{4}(a)-(1+a)^p-(1-a)^p\equiv4-4pK_{p,4,0}(a)-(1+a)^p-(1-a)^p\pmod {p^2}$.
       Thus
       \begin{equation}\label{vp-2}
     \frac{v_p-2}{p}\equiv-\sum\limits_{k=1}^{[\frac p4]}\frac{a^{4k}}{k}-(1+a)q_p(1+a)-(1-a)q_p(1-a)\pmod p.
     \end{equation}
     If $p\equiv1\pmod4$, by Theorem \ref{m4} and Lemma \ref{K} we have
      $-2au_p=\Delta_4(3,p)=4\left[\begin{array}{c}p \\ 3\\\end{array}\right] _{4}(a)-(1+a)^p+(1-a)^p\equiv-4pK_{p,4,3}(a)-(1+a)^p+(1-a)^p\pmod {p^2}$.
      So we have
      \begin{equation}\label{up-1}
        \frac{u_p-1}{p}\equiv\frac{1}{2a}\left(4K_{p,4,3}(a)+(1+a)q_p(1+a)-(1-a)q_p(1-a)\right)\pmod p.
          \end{equation}
         By Lemma \ref{lucasbasic}, we have $u_{p-1}=\frac{1}{-4a^2}(2v_p-2v_{p-1})=\frac{1}{2a^2}(v_{p-1}-v_p)$ and $u_{p-1}=u_p-\frac12v_{p-1}$. By Lemma \ref{quotient-v} and Eq.(\ref{vp-2}), we have
       $$ \begin{array}{l}
         \frac{u_{p-1}}{p}=\frac{1}{2a^2}\left( \frac{v_{p-1}-2}p-\frac{v_{p}-2}p\right)\\
         \equiv\frac{1}{2a^2}\left(\sum\limits_{k=1}^{\frac{p-1}{4}}\frac{a^{4k}}{k}+(1+a)q_p(1+a)+(1-a)q_p(1-a)+q_p(a^2+1)\right)\pmod p.
        \end{array}$$
        Similarly, by Lemma \ref{quotient-v} and Eq.(\ref{up-1}), we have
         \begin{align*}
         \frac{u_{p-1}}{p}
         = &\frac{u_p-1}p-\frac{1}{2}\frac{v_{p-1}-2}p\\
         \equiv&\frac{1}{2a}\left(-4\sum\limits_{k=1}^{\frac{p-1}{4}}\frac{a^{4k-1}}{4k-1}+(1+a)q_p(1+a)-(1-a)q_p(1-a)\right)\\
         &-\frac{1}{2}q_p(a^2+1)\pmod p.
         \end{align*}
         If $p\equiv3\pmod4$, by Theorem \ref{m4} and Lemma \ref{K} we have
         $2au_p=\Delta_4(1,p)=4\left[\begin{array}{c}p \\ 1\\\end{array}\right] _{4}(a)-(1+a)^p+(1-a)^p\equiv-4pK_{p,4,1}(a)-(1+a)^p+(1-a)^p\pmod {p^2}$.
         So we have
      \begin{equation}\label{up+1}
        \frac{u_p+1}{p}\equiv\frac{1}{2a}\left(-4K_{p,4,1}(a)-(1+a)q_p(1+a)+(1-a)q_p(1-a)\right)\pmod p.
          \end{equation}
         By Lemma \ref{lucasbasic}, we have $u_{p+1}=\frac{1}{-4a^2}(2v_{p+1}-2(a^2+1)v_p)=\frac{1}{2a^2}((a^2+1)v_p-v_{p+1})$ and $u_{p+1}=(a^2+1)u_p+\frac12v_{p+1}$. By Lemma \ref{quotient-v} and Eq.(\ref{vp-2}), we have
          $$ \begin{array}{l}
         \frac{u_{p+1}}{p}=\frac{1}{2a^2}\left( (a^2+1)\frac{v_p-2}{p}-\frac{v_{p+1}-2(a^2+1)}{p}\right)\\
         \equiv-\frac{a^2+1}{2a^2}\left(\sum\limits_{k=1}^{\frac{p-3}{4}}\frac{a^{4k}}{k}+(1+a)q_p(1+a)+(1-a)q_p(1-a)\right)\\
         -\frac{a^2+1}{2a^2}q_p(a^2+1)\pmod p.
        \end{array}$$
         Similarly, by Lemma \ref{quotient-v} and Eq.(\ref{up+1}), we have
         \begin{align*}
         \frac{u_{p+1}}{p}
         = &(a^2+1)\frac{u_p+1}p+\frac{1}{2}\frac{v_{p+1}-2(a^2+1)}p\\
         \equiv&\frac{a^2+1}{2a}\left(4\sum\limits_{k=1}^{\frac{p+1}{4}}\frac{a^{4k-3}}{4k-3}-(1+a)q_p(1+a)+(1-a)q_p(1-a)\right)\\
         &+\frac{a^2+1}{2}q_p(a^2+1)\pmod p.
         \end{align*}
 \end{proof}

 Given a value of $a$, by Theorem \ref{4lucas+}, we can obtain a concrete congruence for a specific Lucas quotient. We provide one such example.

  \begin{corollary}
     Let $p>5$ be an odd prime, and $\{u_n\}_{n\geq0}$ be the Lucas sequence defined as
     \[u_0=0,u_1=1,u_{n+1}=2u_n-5u_{n-1}\;\textup{for}\;n\geq1.\]
      Then we have: if $p\equiv1\pmod4$,
      \begin{align*}
     \frac{u_{p-1} }{p}
     \equiv&\frac 18 \sum\limits_{k=1}^{\frac{p-1}{4}}\frac{16^k}{k}+\frac 38q_p(3)+\frac 18q_p(5)\\
     \equiv&-\frac12\sum\limits_{k=1}^{\frac{p-1}{4}}\frac{16^k}{4k-1}+\frac34q_p(3)-\frac12q_p(5)\pmod p;
   \end{align*}
     if $p\equiv3\pmod4$,
     \begin{align*}
     \frac{u_{p+1} }{p}
     \equiv&-\frac 58 \sum\limits_{k=1}^{\frac{p-3}{4}}\frac{16^k}{k}-\frac {15}8q_p(3)-\frac 58q_p(5)\\
     \equiv&\frac58\sum\limits_{k=1}^{\frac{p+1}{4}}\frac{16^k}{4k-3}-\frac{15}4q_p(3)+\frac52q_p(5)\pmod p.
   \end{align*}
  \end{corollary}

  \begin{proof}
    Set $a=-2$ in Theorem \ref{4lucas+}.
  \end{proof}

 \section{$\Delta_6(r,n)$}
In this section, we consider the calculation of $\Delta_6(r,n)$.
 \begin{theorem}\label{m6}
  Let $\{V_n\}_{n\geq0}$ be the Lucas sequences defined as
   $$V_0=2,\;V_1=a+2,\;V_{n+1}=(a+2)V_n-(a^2+a+1)V_{n-1}\mbox{ for }n\geq1.$$
  Let $\{v_n\}_{n\geq0}$ be the Lucas sequences defined as
   $$v_0=2,\;v_1=2-a,\;v_{n+1}=(2-a)v_n-(a^2-a+1)v_{n-1}\mbox{ for }n\geq1.$$
  Then we have, for $n\geq1$,
  \begin{align*}
  \Delta_6(0,n)&=V_n+v_n,\\
  \Delta_6(1,n)&=-\frac 1aV_n+\frac{a^2+a+1}{a}V_{n-1}-\frac 1av_n+\frac{a^2-a+1}{a}v_{n-1},\\
  \Delta_6(2,n)&=-\frac {a+1}aV_n+\frac{a^2+a+1}{a}V_{n-1}-\frac {a-1}av_n-\frac{a^2-a+1}{a}v_{n-1},\\
  \Delta_6(3,n)&=-V_n+v_n,\\
  \Delta_6(4,n)&=\frac {1}aV_n-\frac{a^2+a+1}{a}V_{n-1}-\frac {1}av_n+\frac{a^2-a+1}{a}v_{n-1},\\
  \Delta_6(5,n)&=\frac {a+1}aV_n-\frac{a^2+a+1}{a}V_{n-1}-\frac {a-1}av_n-\frac{a^2-a+1}{a}v_{n-1}.\\
  \end{align*}
\end{theorem}

\begin{proof}
   Since $V_2=-a^2+2a+2,v_2=-a^2-2a+2,V_3=-2a^3-3a^2+3a+2,v_3=2a^3-3a^2-3a+2,V_4=-a^4-8a^3-6a^2+4a+2,v_4=-a^4+8a^3-6a^2-4a+2$,
   one can verify the following simple facts:
   \begin{align*}
   &\Delta_6(0,1)=4=V_1+v_1,\\
   &\Delta_6(0,2)=-2a^2+4=V_2+v_2,\\
   &\Delta_6(0,3)=-6a^2+4=V_3+v_3,\\
   &\Delta_6(0,4)=-2a^4-12a^2+4=V_4+v_4;\\
   &\Delta_6(1,1)=4a=-\frac 1aV_1+\frac{a^2+a+1}{a}V_0-\frac 1av_1+\frac{a^2-a+1}{a}v_0,\\
   &\Delta_6(1,2)=8a=-\frac 1aV_2+\frac{a^2+a+1}{a}V_1-\frac 1av_2+\frac{a^2-a+1}{a}v_1,\\
   &\Delta_6(1,3)=-2a^3+12a=-\frac 1aV_3+\frac{a^2+a+1}{a}V_2-\frac 1av_3+\frac{a^2-a+1}{a}v_2,\\
   &\Delta_6(1,4)=-8a^3+16a=-\frac 1aV_4+\frac{a^2+a+1}{a}V_3-\frac 1av_4+\frac{a^2-a+1}{a}v_3.
   \end{align*}
   By Theorem \ref{evenm}, we have, for $n\geq1$,
   \begin{align*}
    \Delta_6(r,n+4)=&4\Delta_6(r,n+3)-(a^2+6)\Delta_6(r,n+2)\\
    &+(2a^2+4)\Delta_6(r,n+1)-(a^4+a^2+1)\Delta_6(r,n).
     \end{align*}
   Thus we can prove the theorem by induction on $n$ for $r=0,1$.

   It is easy to see that $\Delta_6(r,n)+\Delta_6(r+2,n)+\Delta_6(r+4,n)=3\Delta_2(r,n)=0$ and $\Delta_6(r,n)+\Delta_6(r+3,n)=2\Delta_3(r,n)$.
   Hence $\Delta_6(2,n)=-\Delta_6(0,n)-\Delta_6(4,n)=-\Delta_6(0,n)+\Delta_6(1,n)-2\Delta_3(1,n)$. Thus, by Remark \ref{m3v} and the formulae for
   $\Delta_6(0,n)$ and $\Delta_6(1,n)$, we can obtain the formula for $\Delta_6(2,n)$. Finally, by Remark \ref{m3v} and the formulae for
   $\Delta_6(0,n),\Delta_6(1,n)$ and $\Delta_6(2,n)$, we can obtain the formulae for $\Delta_6(3,n),\Delta_6(4,n)$ and $\Delta_6(5,n)$.
\end{proof}

Note that, for $m=6$, since the Lucas sequences in Theorem \ref{m6} have already appeared in Theorem \ref{m3}, we can not obtain any new Lucas quotient.

\vspace{0.5cm}

\noindent \textbf{Acknowledgments}\quad The work of this paper
was supported by the NNSF of China (Grant No. 11471314), and the National Center for Mathematics and Interdisciplinary Sciences, CAS.

\end{document}